\documentclass[11pt]{amsart}
\usepackage{amsmath}
\usepackage{amssymb}
\usepackage{amsfonts}
\usepackage{mathrsfs}
\usepackage{graphicx}
\usepackage{epsfig}
\usepackage[T1]{fontenc}
\numberwithin{equation}{section}       % Number formulas within sections

\theoremstyle{plain}
\newtheorem{theorem}{Theorem}[section]

\newtheorem{prop}{Proposition}[section]

\newtheorem{lemma}[prop]{Lemma}

\theoremstyle{definition}

\theoremstyle{remark}

\newtheoremstyle{citing}% name
  {3pt}%      Space above, empty = `usual value'
  {3pt}%      Space below
  {\itshape}% Body font
  {}%         Indent amount (empty = no indent, \parindent = para indent)
  {\bfseries}% Thm head font
  {.}%        Punctuation after thm head
  {.5em}%     Space after thm head: " " = normal interword space;
  {\thmnote{#3}}% Thm head spec

\theoremstyle{citing}
\DeclareMathAlphabet{\mathpzc}{OT1}{pzc}{m}{it} % Zapf Chancery math alphabet

\newcommand{\teta}{\widetilde{\teta}}

\begin{document}
\title[]{On the Laurent coefficients of the Riemann map for the complement of the
Mandelbrot set}
\author{Genadi Levin}
\address{Institute of Mathematics, The Hebrew University, 
Jerusalem, 91904, Israel}
\email{levin@math.huji.ac.il}
%\thanks{The second author is partially supported by Polish MNiSW Grant N N201 607640}
\date{\today}
\begin{abstract}
Let $\psi(w)=w+C_0+\frac{C_1}{w}+...+\frac{C_\ell}{w^\ell}+...$ be the 
Laurent series of the Riemann uniformization map from the complement of the
unit disk onto the complement of the Mandelbrot set. 
We straighten a result of~\cite{l2} about arithmetic properties of the coefficients $C_\ell$. This confirms an empirical observation by Don Zagier, 
see~\cite{bfh}.  
\end{abstract}

\maketitle

\section{Introduction and the main result}
Let $f_c(z)=z^2-c$ and $f^n_c$ denotes its $n$-iterate, $n=1,2,...$. 
(Here and later on we keep the notations of~\cite{l2}, in particular,
we use the parameter 
$-c$ instead of standard $c$.) The set
$M=\{c: \sup_{n}|f_c^n(0)|<\infty \}$ is called the Mandelbrot set.
It is a closed bounded subset of the plane and, by the maximum principle, the complement 
${\bf C}\setminus M=\{c: f_c^n(0)\to \infty, n\to \infty\}$
is connected. Douady and Hubbard~\cite{dh} prove that $M$ is also connected. 
For the proof, they construct
a conformal isomorphism 
\begin{equation}\label{phi}
\varphi:{\bf C}\setminus M \to {\bf C}\setminus \overline{\bf D}
\end{equation}  
from the complement of $M$ onto the complement of the unit disk. 
%is constructed~\cite{dh}. 
%This map and its inverse are the objects of this note.
%One can normalize $\varphi$ in such a way that $\varphi(c)/c\to 1$ as $c\to \infty$.
Note that the famous MLC conjecture says that $M$ is locally connected and 
this is equivalent to say
that the inverse (Riemann) map
\begin{equation}\label{psi}
\psi=\varphi^{-1}: {\bf C}\setminus \overline{\bf D}\to {\bf C}\setminus M
\end{equation}
extends continuously onto the unit circle $\partial {\bf D}$.

Consider the Laurent series of $\varphi$ and $\psi$ at $\infty$:
\begin{equation}\label{B}
\varphi(c)=c+B_0+\frac{B_1}{c}+...+\frac{B_\ell}{c^\ell}+...,
\end{equation}
\begin{equation}\label{C}
\psi(w)=w+C_0+\frac{C_1}{w}+...+\frac{C_\ell}{w^\ell}+... .
\end{equation}
Every number $B_\ell$, $C_\ell$ is either zero or
has a finite 2-adic expansion: 
\begin{equation}\label{exp}
B_\ell=\frac{R_\ell}{2^{p_\ell}}, \ C_\ell=\frac{K_\ell}{2^{q_\ell}}, \ \ell=0,1,...,
\end{equation}
for some odd integer numbers $R_\ell$, $K_\ell$ and some integers 
$p_\ell$, $q_\ell$. 
In what follows, given an integer $m\not=0$,
we denote by $ord (m)$ the maximal power of $2$ which devides $m$, and we set  $ord (0)=+\infty$.
For a fraction $m/n$ with $m,n$ non-zero integers, $ord (m/n)=ord(m) -ord(n)$. In particular,
$ord (B_\ell)=-p_\ell$, $ord (C_\ell)=-q_\ell$ for non-zero $B_\ell$, $C_\ell$.
The following properties of the Laurent coefficients $B_\ell$ and $C_\ell$ are found in
\cite{l1},~\cite{l2}.
\begin{theorem}\label{ordB} \cite{l1},~\cite{l2}
For every $\ell\ge 0$, 
\begin{equation}\label{ordb}
p_\ell=\ell+1 + ord (\ell+1)!=ord(2\ell+2)!
\end{equation}
In particular, $B_\ell\not=0$.
\end{theorem}
\begin{theorem}\label{zeroC} \cite{l1},~\cite{l2}
For $m\ge 2$ and $0\le l\le 2^{m}-3$, $C_{(2l+1)2^{m}}=0$. 
\end{theorem}
\begin{theorem}\label{ordC} \cite{l2}
For every odd $\ell$, 
\begin{equation}\label{ordc}
q_\ell=p_\ell=\ell+1 + ord (\ell+1)!
\end{equation}
In particular, $C_\ell\not=0$, if $\ell$ is odd.
\end{theorem}
Theorems~\ref{ordB}-\ref{zeroC} are announced in~\cite{l1}
along with schemes of their proofs.
Detailed proofs of Theorem~\ref{zeroC} (in a more general setting 
of the Multibrot set) 
and Theorem~\ref{ordB} appear in~\cite{l2}. 
Theorem~\ref{ordC} is only stated in~\cite{l2}
with the note that its proof is analogous to the one of Theorem~\ref{ordB}. 
Recently, the author learned from~\cite{bfh} that
Theorem~\ref{zeroC} and Theorem~\ref{ordC} are parts of empirical 
observations (based on computer
calculations) by Don Zagier: see the conjectures (i)-(iii) on
p.32-33 of~\cite{bfh}. Since the proof of Theorem~\ref{ordC}
was not written in~\cite{l2} we reproduce it in the present note
as a consequence of Theorem~\ref{ordB}. In the
course of the proof we confirm also all parts of 
Zagier's observation (ii), see Theorem~\ref{Zager} below.

A particular case of Theorem~\ref{zeroC} (for $l=0$) 
firts appeared in~\cite{j} 
(at the time of writing~\cite{l1}-\cite{l2} the author had no access 
to~\cite{j} though).
For a proof of Theorem~\ref{zeroC} which is different 
from~\cite{l1},~\cite{l2}, see~\cite{bfh}.

Let us state the main result. It is about
(formal) inverse of any 
(formal) Laurent series with rational coefficients such that 
their denominators $2^{p_\ell}$ satisfy~(\ref{ordb}):
\begin{theorem}\label{Zager} 
Suppose that~(\ref{B}) is a formal Laurent series
and~(\ref{C}) is its formal inverse.
Assume that the coefficients in~(\ref{B}) are of the form
$B_\ell=\frac{R_\ell}{2^{p_\ell}}$, where
$R_\ell$ are odd integer numbers and the numbers $p_\ell$ are given
by the formula~(\ref{ordb}), $\ell=0,1,....$ 
Then, for any $\ell$, the coefficient $C_\ell$ in~(\ref{C}) is either zero
or of the form 
$C_\ell=\frac{K_\ell}{2^{q_\ell}}$, where $K_\ell$ is an odd integer and,
for the integer $q_\ell$,
\begin{equation}\label{zager}
q_\ell
\le p_\ell=\ell+1 + ord (\ell+1)!
\end{equation}
The equality in~(\ref{zager}) holds if and only if either $\ell=0$ or 
$\ell$ is odd.
\end{theorem}
{\bf Acknowledgment.} The author thanks Dierk Schleicher for sending him~\cite{bfh}.
\section{Proof of Theorem~\ref{Zager}}
We use standard notations
$$\binom{n}{k}=\frac{n!}{k!(n-k)!}, \ \ 0\le k\le n$$
for binomial coefficients, and
$$\binom{k}{k_1, k_2, ... ,k_j}=\frac{k!}{k_1!k_2!...k_j!}, \ \ 0\le k_1,...,k_j\le k, \ k_1+...+k_j=k$$ 
for multinomial ones.
\begin{lemma}\label{Rec}
$C_0=-B_0$. For every $\ell>0$,
\begin{equation}\label{rec}
C_\ell=-\sum_{k=0}^{\ell-1} C_k M_{\ell-k} - (M_{\ell+1}-P_\ell),
\end{equation}
where 
\begin{equation}\label{m}
M_k=\sum_{j=1}^k \binom{k}{j}\sum_{(i_1,...,i_j):\sum_t (i_t+1)=k}B_{i_1}...B_{i_j},
\end{equation}
\begin{equation}\label{p}
P_k=\sum_{j=1}^k \binom{k}{j}\sum_{(i_1,...,i_j):\sum_t (i_t+1)=k+1}B_{i_1}...B_{i_j},
\end{equation}
\end{lemma}
\begin{proof}
That follows from a comparison of the free terms in the formal identity
$$\varphi(c)^{\ell+1}+C_0\varphi(c)^{\ell}+...+C_{\ell-1}\varphi(c)+C_\ell+
O(\frac{1}{c})=c\varphi(c)^{\ell}.$$
\end{proof}

\begin{lemma}\label{Beq}
Let $i_t$, $t=1,...,j$, be non-negative integers. Denote 
$I=i_1+...+t_j$. Then we have:
\begin{equation}\label{beq}
ord (B_{i_1}...B_{i_j})\ge ord (B_{I+j-1}).
\end{equation}
The equality in~(\ref{beq}) holds if and only if 
the (integer)
$\binom{I+j}{i_1+1, i_2+1, ... ,i_j+1}$ is an odd number.
\end{lemma}
\begin{proof} By~(\ref{ordb}), 
$ord(B_{i_1}...B_{i_j})=\sum_{t=1}^jord(B_{i_t})=-(I+j)-ord\Pi_{t=1}^j(i_t+1)!$
and $ord(B_{I+j-1})=-(I+j)-ord(I+j)!$. But $ord(I+j)!-ord\Pi_{t=1}^j(i_t+1)!=
ord\binom{I+j}{i_1+1,...,i_j+1}$, where $\binom{I+j}{i_1+1,...,i_j+1}$ is integer.
Hence,  $ord(I+j)!-ord\Pi_{t=1}^j(i_t+1)!\ge 0$ and the equality holds if and only if 
$\binom{I+j}{i_1+1,...,i_j+1}$ is odd.
\end{proof}
We will use repeatedly the following obvious 

{\bf Fact.}{\it If $D=\sum_{k=1}^{k_0} D_k$ is a sum of fractions 
$D_k=\frac{L_k}{2^{r_k}}$, where $L_k$ is an odd integer, $p_k$ is an integer
and $p_k\le p_0$ for every $k$, then $ord(D)\ge -p_0$. The equality
$ord(D)=-p_0$ holds if and only if $p_k=p_0$ for an odd number of 
indexes $k$.}
 
We need also two identities (a)-(b):

(a) For any integers $1\le m\le n$,
\begin{equation}\label{id1}
\sum_{j=1}^m \binom{m}{m-j}\sum_{(r_1,...,r_j): r_t\ge 1, 1\le t\le j,
r_1+...+r_j=n} \binom{n}{r_1, r_2, ... ,r_j}=m^n.
\end{equation}

(b) For any integer $n\ge 1$,
\begin{equation}\label{id2}
\sum_{i=0}^{[n/2]} \binom{n}{2i}=2^{n-1}.
\end{equation}
\begin{lemma}\label{key}

(a) For every $k>0$, $ord(M_k)\ge ord(B_{k-1})$ and the equality holds
if and only if $k$ is an odd number.

(b)  For every $k>0$, $ord(P_k)\ge ord(B_{k})$ and the equality holds
if and only if $k$ is an odd number.
\end{lemma}
\begin{proof} (a) We use~(\ref{m}). By Lemma~\ref{Beq}, for every multi-index
$(i_1,...,i_j)$ as in~(\ref{m}), $ord(B_{i_1}...B_{i_j})\ge ord(B_{k-1})$
and the equality holds if and only if $\binom{k}{i_1+1,...,i_j+1}$ is odd.
Hence, $ord(\sum_{(i_1,...,i_j):\sum_t (i_t+1)=k}B_{i_1}...B_{i_j})\ge ord(B_{k-1})$
and the equality holds if and only if the sum
$$\sum_{(r_1,...,r_j): r_t\ge 1, 1\le t\le j,r_1+...+r_j=k} \binom{k}{r_1, r_2, ... ,r_j}$$
is odd. Then, by~(\ref{m}), $ord(M_k)\ge ord(B_{k-1})$ 
and the equality holds if and 
only if the sum 
$$\sum_{j=1}^k \binom{k}{j}\sum_{(r_1,...,r_j): r_t\ge 1, 1\le t\le j,
r_1+...+r_j=k} \binom{k}{r_1, r_2, ..., r_j}$$
is odd. By the identity~(\ref{id1}) with $m=n=k$, 
the latter sum is equal to $k^k$.
This proves (a). The proof of (b) is similar using~(\ref{p}) and~(\ref{id1}) where we put 
$m=k$ and $n=k+1$.
\end{proof}
Now we prove Theorem~\ref{Zager} by indunction on $\ell\ge 1$. As $C_1=-B_1$,
the statement is true for $\ell=1$. Assume that, for some $\ell>1$,
$ord(C_k)>ord(B_k)$ for any even $2\le k\le \ell-1$,
and $ord(C_k)=ord(B_k)$ for any odd $1\le k\le \ell-1$.
We show in (i)-(iv), see below, 
that then $ord(C_k)>ord(B_k)$, if $\ell$ is even and
$ord(C_k)=ord(B_k)$, if $\ell$ is odd.  

(i). By Lemma~\ref{key}, $ord(M_{\ell+1}-P_\ell)=ord(B_\ell)$ as the integers $(\ell+1)^{\ell+1}$ and $\ell^{\ell+1}$
have different parity.

(ii). $ord(C_0 M_\ell)=ord(B_0 M_\ell)=-1+ord(M_\ell)\ge
-1+ord(B_{\ell-1})=-(1+\ell)-ord(\ell!)$ and the equality holds if and only if
$\ell$ is odd. In turn, $-(1+\ell)-ord(\ell!)\ge -(1+\ell)-ord(\ell+1)!=
ord(B_\ell)$, where the equality holds if and only if $\ell$ is even. Thus
in any case, $ord(C_0 M_\ell)>ord(B_\ell)$.

(iii). Let $\ell$ be an odd number. Then, for each odd $1\le k\le \ell-1$,
$\ell-k$ is even, hence, $ord(M_{\ell-k})>ord(B_{\ell-k-1})$.
Also, $ord(C_k)\ge ord(B_k)$ by the induction hypothesis.
For each even $1\le k\le \ell-1$, $\ell-k$ is odd, hence,
$ord(M_{\ell-k})=ord(B_{\ell-k-1})$, but, by the induction hypothesis,
$ord(C_k)>ord(B_k)$. Therefore, for each $1\le k\le \ell-1$,
$ord(C_k M_{\ell-k})>ord(B_k B_{\ell-k-1})\ge ord(B_\ell)$ and using also (ii),
$ord(C_0 M_\ell+\sum_{k=1}^{\ell-1} C_k M_{\ell-k})>ord(B_\ell)$. Thus, by (i),
$$ord(C_\ell)=ord(M_{\ell+1}-P_\ell+\sum_{k=0}^{\ell-1} C_k M_{\ell-k})=ord(B_\ell).$$

(iv). Let $\ell$ be even. For even $1\le k\le \ell-1$, 
by the induction hypotheis, $ord(C_k)>ord(B_k)$, hence,
$ord(C_k M_{e\\-k})>ord(B_k B_{\ell-k-1})\ge ord(B_\ell)$.
For odd  $1\le k\le \ell-1$, by the induction hypotheis, $ord(C_k)=ord(B_k)$
and since $\ell-k$ is odd too, $ord (M_{\ell-k})=ord (B_{\ell-k-1})$. Hence,
for odd $k$, $ord (C_k M_{\ell-k})=ord (B_k B_{\ell-k-1})\ge ord(B_\ell)$
and the equality holds if and only if $\binom{\ell+1}{k+1}$ is odd.
By the identity~(\ref{id2}), $\sum_{k=1, k \text{ is odd }}^{\ell-1}
\binom{\ell+1}{k+1}=2^\ell-1$ (we use that $\ell+1$ is odd).
As $2^\ell-1$ is even, that means that in the sum
$\sum_{k=0}^{\ell-1}C_k M_{\ell-k}$, 
we have: $ord (C_k M_{\ell-k})\ge ord (B_\ell)$ 
for every $k$ and the number of indexes where the equality holds is odd. Hence,
$ord(\sum_{k=0}^{\ell-1}C_k M_{\ell-k})=ord(B_\ell)$. Using (i) we have
in the considered case: $ord(M_{\ell+1}-P_\ell)=ord(\sum_{k=0}^{\ell-1}C_k M_{\ell-k})=ord(B_\ell)$.
Therefore,
$$ord(C_\ell)=ord(M_{\ell+1}-P_\ell+\sum_{k=0}^{\ell-1}C_k M_{\ell-k})>ord(B_\ell).$$
This completes the step of induction.

\end{document}